\documentclass[10pt]{amsart}

\usepackage[T1]{fontenc}
\usepackage{amsmath,amssymb,mathtools}
\usepackage{hyperref}

\hypersetup{hidelinks}
\allowdisplaybreaks
\numberwithin{equation}{section}

\newtheorem{theorem}{Theorem}[section]
\newtheorem{proposition}[theorem]{Proposition}
\newtheorem{lemma}[theorem]{Lemma}
\newtheorem{corollary}[theorem]{Corollary}
\theoremstyle{definition}

\theoremstyle{remark}
\newtheorem{remark}[theorem]{Remark}

\title[Equality cases for bounded-rank commutators]
{Equality cases for matrix spaces with bounded-rank commutators}
\author[Zhi-Lin Zhang]{Zhi-Lin Zhang}

\address{Independent Researcher, Taipei, Taiwan}

\email{hsa00000@gmail.com}
\date{}

\begin{document}

\begin{abstract}
Let \(0\leq k<n\), and let \(\mathcal V\subseteq M_n(\mathbb C)\) be a complex
linear subspace satisfying \(\operatorname{rank}[S,T]\leq k\) for all
\(S,T\in\mathcal V\). Omladi\v{c}, Radjavi, and \v{S}ivic proved the sharp
bound
\(\dim\mathcal V\leq nk+\left\lfloor (n-k)^2/4\right\rfloor+1\)
and conjectured a classification of the equality cases. We prove their
conjecture. If equality holds, then, after a similarity and possibly
transposition, \(\mathcal V\) consists of all block upper-triangular matrices
with arbitrary upper-left and upper-right blocks and with lower-right block
in a maximal-dimensional commuting subspace of \(M_{n-k}(\mathbb C)\).
For \(n-k\geq4\), these commuting subspaces are the classical equality cases
in Schur's theorem; in dimensions \(2\) and \(3\), the additional equality
cases also occur.

At the equality dimension, the rank condition defines a projective algebraic
subset of a Grassmannian. For \(2\leq k\leq n-2\), we determine all of its
irreducible components. If \(n-k\geq4\), there are exactly two components when
\(n-k\) is even and exactly four when \(n-k\) is odd; if
\(n-k\in\{2,3\}\), there are exactly two. When \(n-k\geq4\), the Zariski
tangent space at each such space equals the tangent space to its conjugacy
orbit. When \(n-k\in\{2,3\}\), the two components are obtained by varying the
invariant \(k\)-dimensional subspace and the maximal-dimensional commuting
subspace on the quotient, together with their transposes.
\end{abstract}

\maketitle

\section{Introduction}

The case $k=0$ is the classical problem of determining the largest possible
dimension of a commuting subspace of $M_n(\mathbb C)$. Schur proved the
sharp dimension bound and classified the equality cases \cite{Schur}. We
record the precise result used here in
Theorem~\ref{thm:commuting-classification} below.

Omladi\v{c}, Radjavi, and \v{S}ivic \cite{ORS} introduced the corresponding
bounded-rank commutator problem and conjectured a complete classification of
the equality cases \cite[Conjecture~5]{ORS}. Their results needed below are
recorded in Theorem~\ref{thm:ORS-results}.

Put
\[
m=n-k,
\qquad
s_m=\left\lfloor\frac{m^2}{4}\right\rfloor+1,
\qquad
d_{n,k}=nk+s_m.
\]
For $m\geq2$, let
\begin{equation}\label{eq:ab-choices}
(a,b)\in
\left\{
\left(
\left\lfloor\frac m2\right\rfloor,
\left\lceil\frac m2\right\rceil
\right),
\left(
\left\lceil\frac m2\right\rceil,
\left\lfloor\frac m2\right\rfloor
\right)
\right\}.
\end{equation}
Define
\[
\mathcal C_{a,b}
=
\left\{
\begin{pmatrix}
\lambda I_a&D\\
0&\lambda I_b
\end{pmatrix}:
D\in M_{a\times b}(\mathbb C),\ \lambda\in\mathbb C
\right\}.
\]
Thus $\dim\mathcal C_{a,b}=ab+1=s_m$.
For a subspace $\mathcal V\subseteq M_n(\mathbb C)$, write
\[
\mathcal V^{\mathsf T}=\{S^{\mathsf T}:S\in\mathcal V\}.
\]

We use the following form of Schur's theorem.

\begin{theorem}[Commuting subspaces of maximum dimension]
\label{thm:commuting-classification}
Let $m\geq1$, and let
$\mathcal C\subseteq M_m(\mathbb C)$ be a commuting linear subspace. Then
\[
\dim\mathcal C\leq s_m.
\]
If $\dim\mathcal C=s_m$, then, up to similarity, $\mathcal C$ is one of the
following:
\begin{enumerate}
\item If $m\geq4$, then
$\mathcal C=\mathcal C_{a,b}$ for one of the choices in
\eqref{eq:ab-choices}.
\item\label{item:C-m3} If $m=3$, then $\mathcal C$ is one of
$\mathcal C_{1,2}$, $\mathcal C_{2,1}$, or
\[
\left\{
\operatorname{diag}(\alpha,\beta,\gamma):
\alpha,\beta,\gamma\in\mathbb C
\right\},
\quad
\left\{
\begin{pmatrix}
\alpha&x&0\\
0&\alpha&0\\
0&0&\beta
\end{pmatrix}:
\alpha,\beta,x\in\mathbb C
\right\},
\quad
\mathbb C[E_{12}+E_{23}].
\]
\item\label{item:C-m2} If $m=2$, then $\mathcal C$ is either
$\mathcal C_{1,1}$ or the diagonal algebra in $M_2(\mathbb C)$.
\item If $m=1$, then $\mathcal C=\mathbb C I_1$.
\end{enumerate}
\end{theorem}

The bounded-rank results of Omladi\v{c}, Radjavi, and \v{S}ivic needed here
are the following consequence of \cite[Theorems~2, 4, 12, and~13]{ORS}.

\begin{theorem}[Omladi\v{c}--Radjavi--\v{S}ivic]
\label{thm:ORS-results}
Let $n\geq1$, let $0\leq k<n$, and let
$\mathcal V\subseteq M_n(\mathbb C)$ be a complex linear subspace satisfying
\[
\operatorname{rank}[S,T]\leq k
\qquad(S,T\in\mathcal V).
\]
Then
\[
\dim\mathcal V\leq d_{n,k}.
\]
If $\dim\mathcal V=d_{n,k}$ and either $k=1$, $k=n-1$, or
$\mathcal V$ is an algebra, then there exists
$P\in\operatorname{GL}_n(\mathbb C)$ such that either
$P^{-1}\mathcal VP$ or $P^{-1}\mathcal V^{\mathsf T}P$ equals
\begin{equation}\label{eq:main-block-space}
\left\{
\begin{pmatrix}
A&B\\
0&C
\end{pmatrix}:
A\in M_k(\mathbb C),\
B\in M_{k\times m}(\mathbb C),\
C\in\mathcal C
\right\},
\end{equation}
where $\mathcal C$ is one of the spaces in
Theorem~\ref{thm:commuting-classification}. When $k=0$, the zero-sized
blocks in \eqref{eq:main-block-space} are omitted.
\end{theorem}

Omladi\v{c}, Radjavi, and \v{S}ivic conjectured that the conclusion of
Theorem~\ref{thm:ORS-results} remains true for every
equality case, without the additional assumptions $k=1$, $k=n-1$, or
$\mathcal V$ being an algebra \cite[Conjecture~5]{ORS}. Our main theorem
proves this conjecture.

\begin{theorem}[Equality classification]\label{thm:main}
Let $n\geq1$, let $0\leq k<n$, and let
$\mathcal V\subseteq M_n(\mathbb C)$ be a complex linear subspace such that
\[
\operatorname{rank}[S,T]\leq k
\qquad(S,T\in\mathcal V),
\qquad
\dim\mathcal V=d_{n,k}.
\]
Then there exists $P\in\operatorname{GL}_n(\mathbb C)$ such that either
$P^{-1}\mathcal VP$ or $P^{-1}\mathcal V^{\mathsf T}P$ equals
\eqref{eq:main-block-space}, where $\mathcal C$ is one of the possibilities
in Theorem~\ref{thm:commuting-classification}.
\end{theorem}

A preliminary version of this work treated the special case $(n,k)=(6,2)$ in
\cite{Zhang62}. The present paper gives a uniform argument for arbitrary
$(n,k)$ and replaces the case-specific tangent-space computation used there by
a reduction to the deformation problem for maximal-dimensional commuting
subspaces on the quotient. We also determine the irreducible components of
the equality locus in the Grassmannian and the Zariski tangent spaces at the
points $\mathcal A_{k,a,b}$ defined below.

The case $k=0$ is covered by
Theorem~\ref{thm:commuting-classification}, and the cases
$k=1$ and $k=n-1$ are covered by
Theorem~\ref{thm:ORS-results}. Accordingly, the proof
below concerns the remaining range
\begin{equation}\label{eq:standing-range}
2\leq k\leq n-2,
\qquad
m\geq2.
\end{equation}
The boundary cases are restored in the final proof of
Theorem~\ref{thm:main}.

For the geometric formulation, define
\[
\mathfrak X_{n,k}
=
\left\{
\mathcal V\in\operatorname{Gr}(d_{n,k},M_n(\mathbb C)):
\operatorname{rank}[S,T]\leq k
\text{ for all }S,T\in\mathcal V
\right\}.
\]
For either choice $(a,b)$ in \eqref{eq:ab-choices}, define
\begin{equation}\label{eq:A-kab}
\mathcal A_{k,a,b}
=
\left\{
\begin{pmatrix}
A&B\\
0&C
\end{pmatrix}:
A\in M_k(\mathbb C),\ B\in M_{k\times m}(\mathbb C),\
C\in\mathcal C_{a,b}
\right\}.
\end{equation}
The lower-right block in \eqref{eq:A-kab} is taken relative to
$\mathbb C^m=\mathbb C^a\oplus\mathbb C^b$. Thus
$\mathcal A_{k,a,b}$ is the space in \eqref{eq:main-block-space} obtained
by taking $\mathcal C=\mathcal C_{a,b}$. Since the lower-right block of
every commutator in $\mathcal A_{k,a,b}$ is zero,
$\mathcal A_{k,a,b}\in\mathfrak X_{n,k}$.

For $m\geq2$, let
\[
\mathfrak C_m
=
\left\{
\mathcal C\in\operatorname{Gr}(s_m,M_m(\mathbb C)):
[C,D]=0\text{ for all }C,D\in\mathcal C
\right\}.
\]
If $\mathcal E\subseteq\mathbb C^n$ is a $k$-plane, put
$\mathcal Q_{\mathcal E}=\mathbb C^n/\mathcal E$. For an
$s_m$-dimensional commuting subspace
$\mathcal C\subseteq\operatorname{End}(\mathcal Q_{\mathcal E})$, set
\begin{equation}\label{eq:V-EC}
\mathcal V_{\mathcal E}(\mathcal C)
=
\left\{
T\in\operatorname{End}(\mathbb C^n):
T(\mathcal E)\subseteq\mathcal E,
\ \overline T\in\mathcal C
\right\},
\end{equation}
where $\overline T$ denotes the endomorphism induced on
$\mathcal Q_{\mathcal E}$. Let $\mathfrak Y^+_{n,k}$ be the set of all
spaces \eqref{eq:V-EC}, and let
\begin{equation}\label{eq:Y-minus}
\mathfrak Y^-_{n,k}
=
\left\{\mathcal V^{\mathsf T}:\mathcal V\in\mathfrak Y^+_{n,k}\right\}.
\end{equation}

The following geometric theorem is the main result in the range
\eqref{eq:standing-range}.

For integers $0\leq r<s\leq n$, let $\operatorname{Fl}(r,s;n)$ denote the
partial flag variety of pairs $U\subset H\subset\mathbb C^n$ with
$\dim U=r$ and $\dim H=s$.

\begin{theorem}[Geometric classification]\label{thm:components-intro}
Assume $2\leq k\leq n-2$.
\begin{enumerate}
\item If $m\geq4$, then, for each choice $(a,b)$ in
\eqref{eq:ab-choices}, the two conjugacy orbits
\[
\operatorname{GL}_n(\mathbb C)\cdot\mathcal A_{k,a,b},
\qquad
\operatorname{GL}_n(\mathbb C)\cdot\mathcal A_{k,a,b}^{\mathsf T}
\]
are irreducible components of $\mathfrak X_{n,k}$, and all these components
are pairwise disjoint. Consequently, $\mathfrak X_{n,k}$ has exactly two
irreducible components when $m$ is even and exactly four when $m$ is odd.
For each choice $(a,b)$, both displayed orbits are isomorphic to
$\operatorname{Fl}(k,k+a;n)$ and have dimension $k(n-k)+ab$. Moreover,
\[
T_{\mathcal A_{k,a,b}}\mathfrak X_{n,k}
=
T_{\mathcal A_{k,a,b}}
\bigl(\operatorname{GL}_n(\mathbb C)\cdot\mathcal A_{k,a,b}\bigr).
\]
\item If $m\in\{2,3\}$, then $\mathfrak X_{n,k}$ has exactly two
irreducible components, namely $\mathfrak Y^+_{n,k}$ and
$\mathfrak Y^-_{n,k}$. These two components are disjoint. Their dimension is
$k(n-k)+2$ for $m=2$ and $k(n-k)+6$ for $m=3$.
\end{enumerate}
\end{theorem}

The proof is organized as follows.
Section~\ref{sec:Xnk} shows that $\mathfrak X_{n,k}$ is closed in the
Grassmannian. We then apply the Borel fixed-point theorem together with
\cite[Proposition~11]{ORS} to show that, up to transposition, every
irreducible component of $\mathfrak X_{n,k}$ contains some
$\mathcal A_{k,a,b}$. Thus it is enough to determine the irreducible
component containing each $\mathcal A_{k,a,b}$.

Section~\ref{sec:Akab} describes the conjugacy orbit of
$\mathcal A_{k,a,b}$ and computes its tangent space. This gives the tangent
directions obtained by conjugating $\mathcal A_{k,a,b}$, which we compare
with all tangent directions in $\mathfrak X_{n,k}$.
Section~\ref{sec:tangent-vectors} studies the tangent space of
$\mathfrak X_{n,k}$ at $\mathcal A_{k,a,b}$ and reduces the remaining
calculation to a linear equation involving the commuting subspace
$\mathcal C_{a,b}$.

Section~\ref{sec:m-ge-4} solves this equation when $n-k\geq4$. In this case,
all tangent directions come from conjugation, so the component containing
$\mathcal A_{k,a,b}$ is its conjugacy orbit.
Section~\ref{sec:m-2-3} treats $n-k=2$ and $n-k=3$. In these cases, the
commuting subspace on the quotient can also vary, so the component containing
$\mathcal A_{k,a,b}$ is larger than its conjugacy orbit.

Section~\ref{sec:components} uses these results to determine the irreducible
components of $\mathfrak X_{n,k}$, proves that the listed components are
distinct, and gives their exact number. This proves
Theorem~\ref{thm:components-intro}. Finally, Section~\ref{sec:main-proof}
proves Theorem~\ref{thm:main}.

Throughout, all algebraic sets and varieties are understood in the classical
reduced sense over $\mathbb C$, and Zariski tangent spaces are defined from
their vanishing ideals in affine charts.

\section[The set X(n,k)]{\texorpdfstring{The set $\mathfrak X_{n,k}$}{The set X(n,k)}}\label{sec:Xnk}

In this section, assume $2\leq k\leq n-2$.

\subsection[Closedness of X(n,k)]{\texorpdfstring{Closedness of $\mathfrak X_{n,k}$}{Closedness of X(n,k)}}

\begin{proposition}\label{prop:X-closed}
Assume $2\leq k\leq n-2$. Then the subset
$\mathfrak X_{n,k}$ is closed in
$\operatorname{Gr}(d_{n,k},M_n(\mathbb C))$. In particular, it is
projective.
\end{proposition}

\begin{proof}
This is \cite[Lemma~7]{ORS} with the fixed dimension equal to $d_{n,k}$.
\end{proof}

\subsection{Borel-fixed points}

We use the following precise consequence of the upper-triangular fixed-point
analysis in \cite[Proposition~11]{ORS}.

\begin{theorem}[Omladi\v{c}--Radjavi--\v{S}ivic]\label{thm:ORS-fixed}
Assume $2\leq k\leq n-2$. Suppose that
$\mathcal V\in\mathfrak X_{n,k}$ is invariant under conjugation
by every invertible upper-triangular matrix. Then, after a similarity and
possibly transposition, $\mathcal V$ equals $\mathcal A_{k,a,b}$ for some choice
$(a,b)$ in \eqref{eq:ab-choices}. In particular, the additional
maximal-dimensional commuting spaces in dimensions $2$ and $3$ do
not occur as such Borel-fixed representatives.
\end{theorem}

\begin{proposition}\label{prop:component-Akab}
Assume $2\leq k\leq n-2$. Let $Z$ be an irreducible
component of $\mathfrak X_{n,k}$, and set
\[
Z^{\mathsf T}
=
\{\mathcal V^{\mathsf T}:\mathcal V\in Z\}.
\]
Then $Z$ is stable under conjugation by $\operatorname{GL}_n(\mathbb C)$.
Moreover, either $Z$ or $Z^{\mathsf T}$ contains a point in the conjugacy
orbit of $\mathcal A_{k,a,b}$ for some choice $(a,b)$ in
\eqref{eq:ab-choices}.
\end{proposition}

\begin{proof}
Transposition is an automorphism of $\mathfrak X_{n,k}$: it is induced by a
linear automorphism of $M_n(\mathbb C)$, and
$[S^{\mathsf T},T^{\mathsf T}]=-[S,T]^{\mathsf T}$. Hence
$Z^{\mathsf T}$ is again an irreducible component.

By Proposition~\ref{prop:X-closed}, $Z$ is projective. Since
$\operatorname{GL}_n(\mathbb C)$ is irreducible and acts on
$\mathfrak X_{n,k}$ by conjugation, the closure of
$\operatorname{GL}_n(\mathbb C)\cdot Z$ is irreducible and contains $Z$.
By maximality of $Z$, it equals $Z$. Thus $Z$ is
$\operatorname{GL}_n(\mathbb C)$-stable, hence $B$-stable for the
upper-triangular Borel subgroup $B\subseteq\operatorname{GL}_n(\mathbb C)$.

The Borel fixed-point theorem \cite[Theorem~10.4]{Borel} therefore gives a
$B$-fixed point $\mathcal V_0\in Z$. Theorem~\ref{thm:ORS-fixed} shows that
either $\mathcal V_0$ or $\mathcal V_0^{\mathsf T}$ lies in the conjugacy
orbit of some $\mathcal A_{k,a,b}$. Hence either $Z$ or $Z^{\mathsf T}$
contains such a point.
\end{proof}

Proposition~\ref{prop:component-Akab} reduces the classification of
$\mathfrak X_{n,k}$ to identifying the irreducible component through each
$\mathcal A_{k,a,b}$. We now study these spaces and their conjugacy orbits.

\section[The spaces A(k,a,b)]{\texorpdfstring{The spaces $\mathcal A_{k,a,b}$}{The spaces A(k,a,b)}}\label{sec:Akab}

In this section, assume $2\leq k\leq n-2$ and fix one of the
choices $(a,b)$ in \eqref{eq:ab-choices}.
For all block computations, choose a decomposition
\[
\mathbb C^n=\mathcal E\oplus\mathcal F\oplus\mathcal G,
\qquad
\dim\mathcal E=k,
\quad
\dim\mathcal F=a,
\quad
\dim\mathcal G=b.
\]
Relative to this decomposition,
\[
\mathcal A_{k,a,b}
=
\left\{
\begin{pmatrix}
A&B&C\\
0&\lambda I_a&D\\
0&0&\lambda I_b
\end{pmatrix}:
\begin{array}{l}
A\in M_k(\mathbb C),\\
B\in M_{k\times a}(\mathbb C),\\
C\in M_{k\times b}(\mathbb C),\\
D\in M_{a\times b}(\mathbb C),\\
\lambda\in\mathbb C
\end{array}
\right\}.
\]
We first recover the intrinsic flag determined by $\mathcal A_{k,a,b}$ and then
compute its conjugacy orbit and orbit tangent space.

\subsection[The conjugacy orbit of A(k,a,b)]{\texorpdfstring{The conjugacy orbit of $\mathcal A_{k,a,b}$}{The conjugacy orbit of A(k,a,b)}}

Let
\[
\mathcal O_{k,a,b}
=
\operatorname{GL}_n(\mathbb C)\cdot\mathcal A_{k,a,b}\subseteq \mathfrak X_{n,k},
\qquad
\mathcal O_{k,a,b}^{\mathsf T}
=
\operatorname{GL}_n(\mathbb C)\cdot\mathcal A_{k,a,b}^{\mathsf T}.
\]

\begin{proposition}\label{prop:orbit-Akab}
Assume $2\leq k\leq n-2$ and let $(a,b)$ be one of the
choices in \eqref{eq:ab-choices}. Then
\[
\mathcal O_{k,a,b}
\cong
\operatorname{Fl}(k,k+a;n),
\qquad
\dim\mathcal O_{k,a,b}=k(n-k)+ab.
\]
In particular,
$\mathcal O_{k,a,b}$ is a closed, smooth, and irreducible subvariety of
$\mathfrak X_{n,k}$.
\end{proposition}

\begin{proof}
The space $\mathcal A_{k,a,b}$ is a subalgebra of
$M_n(\mathbb C)$. Let
\[
J=
\left\{
\begin{pmatrix}
0&B&C\\
0&0&D\\
0&0&0
\end{pmatrix}
\right\}.
\]
The subspace $J$ is a nilpotent two-sided ideal, and
\[
\mathcal A_{k,a,b}/J
\cong
M_k(\mathbb C)\oplus\mathbb C
\]
is semisimple. Hence
\[
J=\operatorname{rad}(\mathcal A_{k,a,b}).
\]
A direct multiplication, using $a\geq1$, gives
\[
J^2=
\left\{
\begin{pmatrix}
0&0&C\\
0&0&0\\
0&0&0
\end{pmatrix}:
C\in M_{k\times b}(\mathbb C)
\right\}.
\]
Thus the characteristic ideal $J^2$ recovers the flag:
\begin{equation}\label{eq:intrinsic-flag-general}
\sum_{R\in J^2}\operatorname{im}R=\mathcal E,
\qquad
\bigcap_{R\in J^2}\ker R=\mathcal E\oplus\mathcal F.
\end{equation}

Let $P\subseteq\operatorname{GL}_n$ be the parabolic subgroup stabilizing
$\mathcal E\subset\mathcal E\oplus\mathcal F$. If
$g\in\operatorname{GL}_n(\mathbb C)$ normalizes $\mathcal A_{k,a,b}$, then
conjugation by $g$ is an algebra automorphism of $\mathcal A_{k,a,b}$. The
Jacobson radical is characteristic under algebra automorphisms, so $g$
preserves $J^2$ and hence the flag in
\eqref{eq:intrinsic-flag-general}. Thus $g\in P(\mathbb C)$.

Conversely, the defining conditions in \eqref{eq:A-kab} can be written
intrinsically as
\[
\mathcal A_{k,a,b}
=
\left\{
S\in\operatorname{End}(\mathbb C^n):
\begin{array}{l}
(\mathcal E,\mathcal E\oplus\mathcal F)
\in\operatorname{Fl}(k,k+a;n),\\
S(\mathcal E)\subseteq\mathcal E,\\
S(\mathcal E\oplus\mathcal F)\subseteq\mathcal E\oplus\mathcal F,\\
\left.S\right|_{(\mathcal E\oplus\mathcal F)/\mathcal E}
=\lambda I_{(\mathcal E\oplus\mathcal F)/\mathcal E},\\
\left.S\right|_{\mathbb C^n/(\mathcal E\oplus\mathcal F)}
=\lambda I_{\mathbb C^n/(\mathcal E\oplus\mathcal F)}
\text{ for some }\lambda\in\mathbb C
\end{array}
\right\}.
\]
These conditions depend only on the flag, so every element of $P(\mathbb C)$
normalizes $\mathcal A_{k,a,b}$. Hence $P(\mathbb C)$ is exactly the
stabilizer of $\mathcal A_{k,a,b}$ under conjugation.

By the orbit--stabilizer theorem for algebraic group actions, the orbit map
induces an isomorphism
\[
\mathcal O_{k,a,b}
\cong
\operatorname{GL}_n/P
\cong
\operatorname{Fl}(k,k+a;n).
\]
The partial flag variety is projective, smooth, and irreducible. The induced
morphism to the Grassmannian is therefore proper, so its image
$\mathcal O_{k,a,b}$ is closed; smoothness and irreducibility follow from the
isomorphism above. Finally, $P$ consists of the block upper-triangular matrices
of block sizes $k,a,b$, and hence
\[
\dim\mathcal O_{k,a,b}=n^2-\dim P=n^2-\bigl(k^2+a^2+b^2+ka+kb+ab\bigr)=k(n-k)+ab.
\]
\end{proof}

\subsection{The tangent space to the conjugacy orbit}

Define
\[
\mathcal W_{k,a,b}
=
\left\{
\begin{pmatrix}
0&0&0\\
P&U&0\\
Q&R&V
\end{pmatrix}:
\begin{array}{l}
P\in M_{a\times k}(\mathbb C),\\
Q\in M_{b\times k}(\mathbb C),\\
R\in M_{b\times a}(\mathbb C),\\
U\in M_a(\mathbb C),\\
V\in M_b(\mathbb C),\\
\operatorname{tr}U+\operatorname{tr}V=0
\end{array}
\right\}.
\]
Then
\begin{equation}\label{eq:A-W-direct-sum}
M_n(\mathbb C)
=
\mathcal A_{k,a,b}\oplus\mathcal W_{k,a,b}.
\end{equation}

Let $\pi_{\mathcal W}:M_n(\mathbb C)\to\mathcal W_{k,a,b}$ denote the
projection associated with \eqref{eq:A-W-direct-sum}. Let
\[
\mathcal U_{\mathcal W}
=
\left\{
\mathcal V\in\operatorname{Gr}(d_{n,k},M_n(\mathbb C)):
M_n(\mathbb C)=\mathcal V\oplus\mathcal W_{k,a,b}
\right\}.
\]
The standard graph map
\begin{equation}\label{eq:graph-chart-map}
\Theta:
\operatorname{Hom}(\mathcal A_{k,a,b},\mathcal W_{k,a,b})
\longrightarrow
\mathcal U_{\mathcal W},
\qquad
\psi\longmapsto
\Gamma_\psi
=
\{S+\psi(S):S\in\mathcal A_{k,a,b}\},
\end{equation}
is an affine-chart isomorphism with
$\Theta(0)=\mathcal A_{k,a,b}$. In particular, it gives the identification
\begin{equation}\label{eq:Grassmann-tangent-chart}
T_{\mathcal A_{k,a,b}}
\operatorname{Gr}(d_{n,k},M_n(\mathbb C))
\cong
\operatorname{Hom}(\mathcal A_{k,a,b},\mathcal W_{k,a,b}).
\end{equation}
For
\[
L\in M_{a\times k}(\mathbb C),
\qquad
M\in M_{b\times k}(\mathbb C),
\qquad
N\in M_{b\times a}(\mathbb C),
\]
set
\[
X_{L,M,N}
=
\begin{pmatrix}
0&0&0\\
L&0&0\\
M&N&0
\end{pmatrix},
\]
and define the linear map
\[
\delta_{L,M,N}:
\mathcal A_{k,a,b}\longrightarrow\mathcal W_{k,a,b},
\qquad
S\longmapsto
\pi_{\mathcal W}([X_{L,M,N},S]).
\]

\begin{proposition}\label{prop:orbit-Akab-tangent}
Assume $2\leq k\leq n-2$ and let $(a,b)$ be one of the
choices in \eqref{eq:ab-choices}. Under
\eqref{eq:Grassmann-tangent-chart},
\begin{equation}\label{eq:orbit-tangent}
T_{\mathcal A_{k,a,b}}\mathcal O_{k,a,b}
=
\left\{
\delta_{L,M,N}:
L\in M_{a\times k}(\mathbb C),\ 
M\in M_{b\times k}(\mathbb C),\ 
N\in M_{b\times a}(\mathbb C)
\right\}.
\end{equation}
\end{proposition}

\begin{proof}
Consider the orbit map
\[
\rho:
\operatorname{GL}_n(\mathbb C)
\longrightarrow
\operatorname{Gr}(d_{n,k},M_n(\mathbb C)),
\qquad
g\longmapsto g\mathcal A_{k,a,b}g^{-1}.
\]
For
\[
X=
\begin{pmatrix}
X_{11}&X_{12}&X_{13}\\
X_{21}&X_{22}&X_{23}\\
X_{31}&X_{32}&X_{33}
\end{pmatrix}
\in M_n(\mathbb C),
\]
write
\[
X=Y+X_{L,M,N},
\]
where
\[
Y=
\begin{pmatrix}
X_{11}&X_{12}&X_{13}\\
0&X_{22}&X_{23}\\
0&0&X_{33}
\end{pmatrix},
\qquad
L=X_{21},\quad
M=X_{31},\quad
N=X_{32}.
\]
By Proposition~\ref{prop:orbit-Akab}, $Y$ lies in the Lie algebra of
$N_{\operatorname{GL}_n}(\mathcal A_{k,a,b})$. Since $\rho$ is constant on
this normalizer,
\[
(d\rho)_{I_n}(Y)=0.
\]
Since $\rho(I_n)=\mathcal A_{k,a,b}\in\mathcal U_{\mathcal W}$, the map
$\Theta^{-1}\circ\rho$ is defined on a neighborhood of $I_n$. For
$S\in\mathcal A_{k,a,b}$, the expansion
\[
(I_n+tX)S(I_n+tX)^{-1}
=
S+t[X,S]+O(t^2)
\]
shows that, in the graph coordinates \eqref{eq:graph-chart-map},
\begin{equation}\label{eq:orbit-map-differential}
d(\Theta^{-1}\circ\rho)_{I_n}(X)(S)
=
\pi_{\mathcal W}([X,S]).
\end{equation}
Thus, under the tangent-space identification
\eqref{eq:Grassmann-tangent-chart}, the differential $(d\rho)_{I_n}(X)$ is
represented by the map on the right-hand side of
\eqref{eq:orbit-map-differential}. Since $X=Y+X_{L,M,N}$ and
$(d\rho)_{I_n}(Y)=0$, linearity gives
\[
(d\rho)_{I_n}(X)(S)
=
\pi_{\mathcal W}([X_{L,M,N},S])
=
\delta_{L,M,N}(S).
\]
Therefore
\[
\operatorname{Im}(d\rho)_{I_n}
=
\left\{
\delta_{L,M,N}:
L\in M_{a\times k}(\mathbb C),\
M\in M_{b\times k}(\mathbb C),\
N\in M_{b\times a}(\mathbb C)
\right\}.
\]
The parametrization $(L,M,N)\mapsto\delta_{L,M,N}$ is injective by a direct
block computation. Hence, by Proposition~\ref{prop:orbit-Akab},
\[
\dim\operatorname{Im}(d\rho)_{I_n}
=
ak+bk+ab
=
k(n-k)+ab
=
\dim T_{\mathcal A_{k,a,b}}\mathcal O_{k,a,b}.
\]
Since
\[
\operatorname{Im}(d\rho)_{I_n}
\subseteq
T_{\mathcal A_{k,a,b}}\mathcal O_{k,a,b},
\]
equality holds, proving \eqref{eq:orbit-tangent}.
\end{proof}

\section[Tangent vectors at A(k,a,b)]{\texorpdfstring{Tangent vectors at $\mathcal A_{k,a,b}$}{Tangent vectors at A(k,a,b)}}\label{sec:tangent-vectors}

In this section, assume $2\leq k\leq n-2$ and fix one of the
choices $(a,b)$ in \eqref{eq:ab-choices}. We shall use the matrices
\begin{equation}\label{eq:AB-star}
A_\ast=\sum_{i=1}^{k-1}E_{i,i+1},\quad
B_\ast=\sum_{i=1}^{k-1}iE_{i+1,i},\quad
H_\ast=[A_\ast,B_\ast]
=\operatorname{diag}(1,\ldots,1,-(k-1)).
\end{equation}
Thus $H_\ast$ is invertible and $B_\ast$ is nilpotent. We also recall the
standard tangent-space formula for matrices of bounded rank.

\begin{lemma}\label{lem:determinantal-tangent}
Assume $2\leq k\leq n-2$. Let
\[
\mathcal D_k=
\{C\in M_n(\mathbb C):\operatorname{rank}C\leq k\}.
\]
If $C_0\in\mathcal D_k$ has rank exactly $k$, then
\[
T_{C_0}\mathcal D_k
=
\{C_1\in M_n(\mathbb C):
C_1(\ker C_0)\subseteq\operatorname{im}C_0\}.
\]
\end{lemma}

\begin{proof}
This is the standard tangent-space formula for determinantal varieties;
see \cite[Example~14.16]{Harris}.
\end{proof}

\subsection{An equation for tangent vectors}

Put $\mathcal Q=\mathcal F\oplus\mathcal G$ and write
$\mathcal A_{k,a,b}$ in two-block form as
\begin{equation}\label{eq:A-two-block}
\mathcal A_{k,a,b}
=
\left\{
S(A,X,C)
=
\begin{pmatrix}
A&X\\
0&C
\end{pmatrix}:
A\in\operatorname{End}(\mathcal E),\ X\in\operatorname{Hom}(\mathcal Q,\mathcal E),
\ C\in\mathcal C_{a,b}
\right\}.
\end{equation}
Define
\[
\mathcal Z_{a,b}
=
\left\{
\begin{pmatrix}
U&0\\
R&V
\end{pmatrix}
:
U\in M_a(\mathbb C),\
R\in M_{b\times a}(\mathbb C),\
V\in M_b(\mathbb C),\
\operatorname{tr}U+\operatorname{tr}V=0
\right\}.
\]
Then
\begin{equation}\label{eq:C-Z-direct-sum}
\operatorname{End}(\mathcal Q)=\mathcal C_{a,b}\oplus\mathcal Z_{a,b}.
\end{equation}

\begin{lemma}\label{lem:tangent-master}
Assume $2\leq k\leq n-2$ and let $(a,b)$ be one of the
choices in \eqref{eq:ab-choices}. Let
\[
\phi\in T_{\mathcal A_{k,a,b}}\mathfrak X_{n,k},
\qquad
\phi(S)=
\begin{pmatrix}
0&0\\
\phi_{21}(S)&\phi_{22}(S)
\end{pmatrix}
\]
in the decomposition $\mathbb C^n=\mathcal E\oplus\mathcal Q$. For
\[
S=S(A,X,C),
\qquad
T=S(A',X',C'),
\]
set
\[
K_{11}=[A,A'],
\qquad
K_{12}=AX'+XC'-A'X-X'C.
\]
If $K_{11}$ is invertible, then
\begin{equation}\label{eq:tangent-master-short}
\Gamma_{22}
=
\Gamma_{21}K_{11}^{-1}K_{12},
\end{equation}
where
\[
\begin{aligned}
\Gamma_{21}
&=
\phi_{21}(S)A'-C'\phi_{21}(S)
+C\phi_{21}(T)-\phi_{21}(T)A,\\
\Gamma_{22}
&=
\phi_{21}(S)X'-\phi_{21}(T)X
+[\phi_{22}(S),C']+[C,\phi_{22}(T)].
\end{aligned}
\]

\end{lemma}

\begin{proof}
Fix $S,T\in\mathcal A_{k,a,b}$ and consider the regular map
\[
\mathcal C_{S,T}:
\operatorname{Hom}(\mathcal A_{k,a,b},\mathcal W_{k,a,b})
\longrightarrow
M_n(\mathbb C),
\qquad
\psi\longmapsto
[S+\psi(S),\,T+\psi(T)].
\]
If
\[
\psi\in
\Theta^{-1}
\left(
\mathfrak X_{n,k}\cap\mathcal U_{\mathcal W}
\right),
\]
then $\Gamma_\psi=\Theta(\psi)$ belongs to $\mathfrak X_{n,k}$, while
$S+\psi(S)$ and $T+\psi(T)$ belong to $\Gamma_\psi$. Hence the defining
rank condition of $\mathfrak X_{n,k}$ gives
\[
\operatorname{rank}
[S+\psi(S),\,T+\psi(T)]
\leq k.
\]
Therefore
\begin{equation}\label{eq:commutator-map-into-Dk}
\mathcal C_{S,T}
\left(
\Theta^{-1}
\left(
\mathfrak X_{n,k}\cap\mathcal U_{\mathcal W}
\right)
\right)
\subseteq
\mathcal D_k.
\end{equation}
Since
\[
\phi
\in
T_0
\Theta^{-1}
\left(
\mathfrak X_{n,k}\cap\mathcal U_{\mathcal W}
\right),
\]
differentiating \eqref{eq:commutator-map-into-Dk} at $0$ gives
\[
(d\mathcal C_{S,T})_0(\phi)
\in
T_{\mathcal C_{S,T}(0)}\mathcal D_k
=
T_{[S,T]}\mathcal D_k.
\]
Moreover,
\[
(d\mathcal C_{S,T})_0(\phi)
=
\left.\frac{d}{dt}\right|_{t=0}
[S+t\phi(S),\,T+t\phi(T)]=
[\phi(S),T]+[S,\phi(T)].
\]
Thus
\begin{equation}\label{eq:first-order-commutator-tangent}
[\phi(S),T]+[S,\phi(T)]
\in
T_{[S,T]}\mathcal D_k.
\end{equation}
Now set
\[
C_0=[S,T]
=
\begin{pmatrix}
K_{11}&K_{12}\\
0&0
\end{pmatrix},
\qquad
C_1=[\phi(S),T]+[S,\phi(T)]
=
\begin{pmatrix}
*&*\\
\Gamma_{21}&\Gamma_{22}
\end{pmatrix},
\]
where direct block multiplication gives
\[
\begin{aligned}
\Gamma_{21}
&=
\phi_{21}(S)A'-C'\phi_{21}(S)
+C\phi_{21}(T)-\phi_{21}(T)A,\\
\Gamma_{22}
&=
\phi_{21}(S)X'-\phi_{21}(T)X
+[\phi_{22}(S),C']+[C,\phi_{22}(T)].
\end{aligned}
\]
If $K_{11}$ is invertible, then $\operatorname{rank}C_0=k$, and
\[
\ker C_0
=
\left\{
\begin{pmatrix}
-K_{11}^{-1}K_{12}v\\
v
\end{pmatrix}
:
v\in\mathcal Q
\right\},
\qquad
\operatorname{im}C_0
=
\mathcal E\oplus0.
\]
By \eqref{eq:first-order-commutator-tangent}, $C_1\in T_{C_0}\mathcal D_k$. 
Lemma~\ref{lem:determinantal-tangent} therefore gives $C_1(\ker C_0)
\subseteq
\operatorname{im}C_0$. Hence, for every $v\in\mathcal Q$, the lower component of
\[
C_1
\begin{pmatrix}
-K_{11}^{-1}K_{12}v\\
v
\end{pmatrix}
\]
must vanish. Thus
\[
-\Gamma_{21}K_{11}^{-1}K_{12}v+\Gamma_{22}v=0
\qquad
(v\in\mathcal Q).
\]
Since this holds for every $v\in\mathcal Q$, $\Gamma_{22}
=
\Gamma_{21}K_{11}^{-1}K_{12}$.
\end{proof}

\subsection{The lower-left block}

\begin{lemma}\label{lem:lower-left-X}
Assume $2\leq k\leq n-2$ and let $(a,b)$ be one of the
choices in \eqref{eq:ab-choices}. Let
\[
\phi\in T_{\mathcal A_{k,a,b}}\mathfrak X_{n,k}.
\]
Under \eqref{eq:Grassmann-tangent-chart}, write
\[
\phi(S)=
\begin{pmatrix}
0&0\\
\phi_{21}(S)&\phi_{22}(S)
\end{pmatrix}
\]
in the decomposition $\mathbb C^n=\mathcal E\oplus\mathcal Q$. Then
\[
\phi_{21}(S(0,X,0))=0
\qquad
(X\in\operatorname{Hom}(\mathcal Q,\mathcal E)).
\]
\end{lemma}

\begin{proof}
Define
\[
\beta(X)=\phi_{21}(S(0,X,0)).
\]
Fix $X\in\operatorname{Hom}(\mathcal Q,\mathcal E)$. For
\[
g\in\operatorname{GL}(\mathcal E),
\qquad
\lambda,t\in\mathbb C,
\]
apply Lemma~\ref{lem:tangent-master} to
\[
S=S(gA_\ast g^{-1},tX,0),
\qquad
T=S(gB_\ast g^{-1},0,\lambda I_{\mathcal Q}).
\]
Here
\[
K_{11}=gH_\ast g^{-1},
\qquad
K_{12}=t(\lambda I_{\mathcal E}-gB_\ast g^{-1})X.
\]
By linearity of $\phi_{21}$, the $t\lambda$-coefficient of
$\Gamma_{21}$ is $-\beta(X)$, while $\Gamma_{22}$ has no
$t^2\lambda^2$-term. Thus comparison of the $t^2\lambda^2$-coefficients in
\eqref{eq:tangent-master-short} gives
\begin{equation}\label{eq:beta-conjugates}
\beta(X)gH_\ast^{-1}g^{-1}X=0
\qquad
(g\in\operatorname{GL}(\mathcal E)).
\end{equation}
Since
\[
\mathfrak{sl}(\mathcal E)\subseteq
\operatorname{span}_{\mathbb C}
\{gH_\ast^{-1}g^{-1}:g\in\operatorname{GL}(\mathcal E)\},
\]
\eqref{eq:beta-conjugates} implies
\[
\beta(X)YX=0
\qquad
(Y\in\mathfrak{sl}(\mathcal E)).
\]
If $X=0$, then $\beta(X)=0$. If $X\neq0$, choose $v\in\mathcal Q$ such
that $Xv\neq0$. Then
\[
\beta(X)Y(Xv)=0
\qquad
(Y\in\mathfrak{sl}(\mathcal E)).
\]
Since $k\geq2$ and $Xv\neq0$,
$\mathfrak{sl}(\mathcal E)\cdot(Xv)=\mathcal E$. Therefore $\beta(X)=0$.
\end{proof}

\begin{proposition}\label{prop:lower-left-block}
Assume $2\leq k\leq n-2$ and let $(a,b)$ be one of the
choices in \eqref{eq:ab-choices}. Let
\[
\phi\in T_{\mathcal A_{k,a,b}}\mathfrak X_{n,k}.
\]
In the two-block decomposition $\mathbb C^n=\mathcal E\oplus\mathcal Q$, write
\[
\phi(S)=
\begin{pmatrix}
0&0\\
\phi_{21}(S)&\phi_{22}(S)
\end{pmatrix}.
\]
Then there is a unique
$L_\phi\in\operatorname{Hom}(\mathcal E,\mathcal Q)$ such that
\begin{equation}\label{eq:phi21-general-statement}
\phi_{21}(S(A,X,C))=L_\phi A-CL_\phi
\end{equation}
for all $A$, $X$, and $C$.
\end{proposition}

\begin{proof}

Via the graph chart \eqref{eq:graph-chart-map}, regard
\[
\phi
\in
T_0\Theta^{-1}
\left(
\mathfrak X_{n,k}\cap\mathcal U_{\mathcal W}
\right)
\subseteq
\operatorname{Hom}(\mathcal A_{k,a,b},\mathcal W_{k,a,b}).
\]
Thus $\phi(S)\in\mathcal W_{k,a,b}$ for every
$S\in\mathcal A_{k,a,b}$. In the two-block decomposition
$\mathbb C^n=\mathcal E\oplus\mathcal Q$, write
\[
\phi(S)
=
\begin{pmatrix}
0&0\\
\phi_{21}(S)&\phi_{22}(S)
\end{pmatrix},
\qquad
\phi_{21}(S)\in\operatorname{Hom}(\mathcal E,\mathcal Q),
\quad
\phi_{22}(S)\in\mathcal Z_{a,b}.
\]
Using the direct-sum decomposition in \eqref{eq:A-two-block}, define
\[
\alpha(A)=\phi_{21}(S(A,0,0)),
\qquad
\beta(X)=\phi_{21}(S(0,X,0)),
\qquad
\gamma(C)=\phi_{21}(S(0,0,C)).
\]
By linearity,
\begin{equation}\label{eq:phi21-alpha-beta-gamma}
\phi_{21}(S(A,X,C))
=
\alpha(A)+\beta(X)+\gamma(C).
\end{equation}

By Lemma~\ref{lem:lower-left-X}, $\beta=0$.

Put
\[
L_\phi=\alpha(I_{\mathcal E}),
\qquad
\Xi_\ast=I_{\mathcal E}-B_\ast,
\qquad
T_\ast=S(B_\ast,0,I_{\mathcal Q}).
\]
Since $B_\ast$ is nilpotent, $\Xi_\ast$ is invertible. Fix $A\in\operatorname{End}(\mathcal E)$ such that
\[
K_{11}=[A,B_\ast]
\]
is invertible, and apply Lemma~\ref{lem:tangent-master} to
\[
S=S(A+sI_{\mathcal E},tX,C),
\qquad
T=T_\ast.
\]
By Lemma~\ref{lem:lower-left-X},
\[
\begin{aligned}
K_{12}&=t\Xi_\ast X,\\
\Gamma_{22}
&=[C,\phi_{22}(T_\ast)]-t\,\phi_{21}(T_\ast)X,\\
\Gamma_{21}
&=-\bigl(\alpha(A)+sL_\phi+\gamma(C)\bigr)\Xi_\ast
+C\phi_{21}(T_\ast)
-\phi_{21}(T_\ast)(A+sI_{\mathcal E}).
\end{aligned}
\]
Comparing the coefficient of $st$ in
\eqref{eq:tangent-master-short} gives
\[
0
=
-\bigl(L_\phi\Xi_\ast+\phi_{21}(T_\ast)\bigr)
K_{11}^{-1}\Xi_\ast X
\qquad
(X\in\operatorname{Hom}(\mathcal Q,\mathcal E)).
\]
Since $K_{11}$ and $\Xi_\ast$ are invertible and $X$ is arbitrary,
\begin{equation}\label{eq:phi21-Tstar}
\phi_{21}(T_\ast)=-L_\phi\Xi_\ast.
\end{equation}
Now set $s=0$ and compare the coefficients of $t$ in
\eqref{eq:tangent-master-short}. Using
\eqref{eq:phi21-Tstar}, we obtain
\[
L_\phi\Xi_\ast X
=
\bigl(
-(\alpha(A)+\gamma(C)+CL_\phi)\Xi_\ast
+L_\phi\Xi_\ast A
\bigr)K_{11}^{-1}\Xi_\ast X
\]
for every $X$. Hence
\[
L_\phi\Xi_\ast
=
\bigl(
-(\alpha(A)+\gamma(C)+CL_\phi)\Xi_\ast
+L_\phi\Xi_\ast A
\bigr)K_{11}^{-1}\Xi_\ast.
\]
Multiplying on the right by $\Xi_\ast^{-1}K_{11}$ gives
\[
(\alpha(A)+\gamma(C)+CL_\phi)\Xi_\ast
=
L_\phi(\Xi_\ast A-K_{11}).
\]
Since
\[
\Xi_\ast A-[A,B_\ast]=A\Xi_\ast,
\]
we obtain
\begin{equation}\label{eq:alpha-gamma-open}
\alpha(A)+\gamma(C)+CL_\phi=L_\phi A
\end{equation}
whenever $[A,B_\ast]$ is invertible. The set
\[
\{A\in\operatorname{End}(\mathcal E):[A,B_\ast]\text{ is invertible}\}
\]
is a nonempty Zariski-open subset of $\operatorname{End}(\mathcal E)$, since
it contains $A_\ast$. Taking $C=0$ in \eqref{eq:alpha-gamma-open} gives
\[
\alpha(A)=L_\phi A
\]
on this open set. Since both sides are linear in $A$, density gives
\[
\alpha(A)=L_\phi A
\qquad
(A\in\operatorname{End}(\mathcal E)).
\]
Substitution into \eqref{eq:alpha-gamma-open} then gives
\[
\gamma(C)=-CL_\phi
\qquad
(C\in\mathcal C_{a,b}).
\]
Together with Lemma~\ref{lem:lower-left-X} and \eqref{eq:phi21-alpha-beta-gamma}, this proves
\eqref{eq:phi21-general-statement}.

The value of the lower-left block at $S(I_{\mathcal E},0,0)$ is $L_\phi$,
so $L_\phi$ is unique.
\end{proof}

\subsection{The lower-right block}

\begin{lemma}\label{lem:centralizer-Cab}
The centralizer of
$\mathcal C_{a,b}$ in $\operatorname{End}(\mathcal Q)$ is
$\mathcal C_{a,b}$.
\end{lemma}

\begin{proof}
Let
\[
Y=
\begin{pmatrix}
Y_{11}&Y_{12}\\
Y_{21}&Y_{22}
\end{pmatrix}
\in\operatorname{End}(\mathcal Q)
\]
centralize $\mathcal C_{a,b}$. Since
\[
\begin{pmatrix}
0&D\\
0&0
\end{pmatrix}
\in\mathcal C_{a,b}
\qquad
(D\in M_{a\times b}(\mathbb C)),
\]
we have
\[
\begin{pmatrix}
0&Y_{11}D\\
0&Y_{21}D
\end{pmatrix}
=
\begin{pmatrix}
DY_{21}&DY_{22}\\
0&0
\end{pmatrix}
\qquad
(D\in M_{a\times b}(\mathbb C)).
\]
Hence
\[
DY_{21}=0,\qquad
Y_{21}D=0,\qquad
Y_{11}D=DY_{22}
\qquad
(D\in M_{a\times b}(\mathbb C)).
\]
Using matrix units gives
\[
Y_{21}=0,
\qquad
Y_{11}=\lambda I_a,
\qquad
Y_{22}=\lambda I_b
\]
for some $\lambda\in\mathbb C$, while $Y_{12}$ is arbitrary. Thus
$Y\in\mathcal C_{a,b}$. The reverse inclusion holds because
$\mathcal C_{a,b}$ is commutative.
\end{proof}

The complement $\mathcal Z_{a,b}$ in \eqref{eq:C-Z-direct-sum} gives the
standard graph-chart identification
\begin{equation}\label{eq:quotient-Grassmann-tangent}
T_{\mathcal C_{a,b}}
\operatorname{Gr}(s_m,M_m(\mathbb C))
\cong
\operatorname{Hom}(\mathcal C_{a,b},\mathcal Z_{a,b}).
\end{equation}
Define
\begin{equation}\label{eq:T-ab-equation}
\mathcal T_{a,b}
=
\left\{
\eta\in\operatorname{Hom}(\mathcal C_{a,b},\mathcal Z_{a,b}):
[\eta(C),C']+[C,\eta(C')]=0
\text{ for all }C,C'\in\mathcal C_{a,b}
\right\}.
\end{equation}
Every $\eta\in\mathcal T_{a,b}$ satisfies
\[
\eta(I_m)=0.
\]
Indeed, taking $C'=I_m$ in \eqref{eq:T-ab-equation} gives
\[
[C,\eta(I_m)]=0
\qquad(C\in\mathcal C_{a,b}).
\]
Thus $\eta(I_m)$ centralizes $\mathcal C_{a,b}$, so
Lemma~\ref{lem:centralizer-Cab} gives
$\eta(I_m)\in\mathcal C_{a,b}$. Since also
$\eta(I_m)\in\mathcal Z_{a,b}$,
\eqref{eq:C-Z-direct-sum} gives $\eta(I_m)=0$.

\begin{proposition}\label{prop:quotient-identities}
Let $\eta\in\mathcal T_{a,b}$.
Then there are unique linear maps
\[
R\colon M_{a\times b}(\mathbb C)\to M_{b\times a}(\mathbb C),
\qquad
U\colon M_{a\times b}(\mathbb C)\to M_a(\mathbb C),
\qquad
V\colon M_{a\times b}(\mathbb C)\to M_b(\mathbb C)
\]
such that
\begin{equation}\label{eq:eta-URV}
\eta
\left(
\begin{pmatrix}
\lambda I_a&D\\
0&\lambda I_b
\end{pmatrix}
\right)
=
\begin{pmatrix}
U(D)&0\\
R(D)&V(D)
\end{pmatrix},
\qquad
\operatorname{tr}U(D)+\operatorname{tr}V(D)=0.
\end{equation}
Moreover, for all $Y,Z\in M_{a\times b}(\mathbb C)$,
\begin{gather}
ZR(Y)=YR(Z),\label{eq:R-left}\\
R(Y)Z=R(Z)Y,\label{eq:R-right}\\
U(Y)Z-ZV(Y)+YV(Z)-U(Z)Y=0.
\label{eq:UV-identity}
\end{gather}
\end{proposition}

\begin{proof}
Every element of $\mathcal C_{a,b}$ has the form
\[
\begin{pmatrix}
\lambda I_a&D\\
0&\lambda I_b
\end{pmatrix}
=
\lambda I_m+
\begin{pmatrix}
0&D\\
0&0
\end{pmatrix}.
\]
By the observation following \eqref{eq:T-ab-equation} and linearity,
$\eta$ therefore depends only on $D$. Since its values lie in
$\mathcal Z_{a,b}$, there are unique linear maps $U$, $R$, and $V$ satisfying
\eqref{eq:eta-URV}.

For
\[
C_Y=
\begin{pmatrix}
0&Y\\
0&0
\end{pmatrix},
\qquad
C_Z=
\begin{pmatrix}
0&Z\\
0&0
\end{pmatrix},
\]
expanding
$[\eta(C_Y),C_Z]+[C_Y,\eta(C_Z)]=0$ gives
\eqref{eq:R-left}--\eqref{eq:UV-identity}.
\end{proof}

\begin{proposition}\label{prop:tangent-splitting}
Assume $2\leq k\leq n-2$ and let $(a,b)$ be one of the
choices in \eqref{eq:ab-choices}. Let
\[
\phi\in T_{\mathcal A_{k,a,b}}\mathfrak X_{n,k}.
\]
Then there exist unique
\[
L_\phi\in\operatorname{Hom}(\mathcal E,\mathcal Q),
\qquad
\eta_\phi\in\mathcal T_{a,b},
\]
such that, for every
\[
S(A,X,C)
=
\begin{pmatrix}
A&X\\
0&C
\end{pmatrix}
\in\mathcal A_{k,a,b},
\]
one has
\[
\phi(S(A,X,C))
=
\pi_{\mathcal W}
\left(
\left[
\begin{pmatrix}
0&0\\
L_\phi&0
\end{pmatrix},
S(A,X,C)
\right]
\right)
+
\begin{pmatrix}
0&0\\
0&\eta_\phi(C)
\end{pmatrix}.
\]
\end{proposition}

\begin{proof}
Let $L_\phi$ be the map given by Proposition~\ref{prop:lower-left-block}. Set
\[
\widehat L_\phi
=
\begin{pmatrix}
0&0\\
L_\phi&0
\end{pmatrix},
\qquad
\delta_{L_\phi}(S)
=
\pi_{\mathcal W}([\widehat L_\phi,S]).
\]
After splitting $L_\phi:\mathcal E\to\mathcal F\oplus\mathcal G$ into its
two block rows, Proposition~\ref{prop:orbit-Akab-tangent} gives
\[
\delta_{L_\phi}
\in
T_{\mathcal A_{k,a,b}}\mathcal O_{k,a,b}
\subseteq
T_{\mathcal A_{k,a,b}}\mathfrak X_{n,k}.
\]
For $S=S(A,X,C)$, the lower-left block of $\delta_{L_\phi}(S)$ is
\[
[\widehat L_\phi,S]_{21}
=
L_\phi A-CL_\phi.
\]
By Proposition~\ref{prop:lower-left-block}, this is exactly the lower-left
block of $\phi(S)$. Hence, with
\[
\phi_{\mathrm{res}}:=\phi-\delta_{L_\phi},
\]
there is a linear map
\[
\psi:\mathcal A_{k,a,b}\longrightarrow\mathcal Z_{a,b}
\]
such that
\[
\phi_{\mathrm{res}}(S)
=
\begin{pmatrix}
0&0\\
0&\psi(S)
\end{pmatrix}
\qquad
(S\in\mathcal A_{k,a,b}).
\]
For
\[
S=S(A,X,C),
\qquad
T=S(A',X',C'),
\]
suppose that $[A,A']$ is invertible. Applying
Lemma~\ref{lem:tangent-master} to $\phi_{\mathrm{res}}$, the corresponding
$\Gamma_{21}$ vanishes, so \eqref{eq:tangent-master-short} gives
\[
[\psi(S),C']+[C,\psi(T)]=0.
\]
The condition
\[
[A,A']\in\operatorname{GL}(\mathcal E)
\]
defines a nonempty Zariski-open subset of
$\operatorname{End}(\mathcal E)^2$, since
$[A_\ast,B_\ast]=H_\ast$ is invertible by \eqref{eq:AB-star}. The displayed
identity is polynomial in the entries of $S$ and $T$ and holds on a
Zariski-dense open subset of the $(A,A')$-variables. Hence it holds for
all $S,T\in\mathcal A_{k,a,b}$, so
\begin{equation}\label{eq:psi-master}
[\psi(S),C']+[C,\psi(T)]=0.
\end{equation}

Now let $S=S(A,X,0)$. Taking $C=0$ in \eqref{eq:psi-master} gives
\[
[\psi(S),C']=0
\qquad
(C'\in\mathcal C_{a,b}).
\]
Hence $\psi(S)$ centralizes $\mathcal C_{a,b}$, so
Lemma~\ref{lem:centralizer-Cab} gives
\[
\psi(S)\in\mathcal C_{a,b}.
\]
Since also $\psi(S)\in\mathcal Z_{a,b}$ and
$\mathcal C_{a,b}\cap\mathcal Z_{a,b}=0$ by
\eqref{eq:C-Z-direct-sum}, we obtain
\[
\psi(S(A,X,0))=0.
\]
Therefore $\psi(S(A,X,C))$ depends only on $C$. Define
\[
\eta_\phi:\mathcal C_{a,b}\longrightarrow\mathcal Z_{a,b},
\qquad
\eta_\phi(C)=\psi(S(0,0,C)).
\]
Then $\eta_\phi$ is linear and
\[
\phi_{\mathrm{res}}(S(A,X,C))
=
\begin{pmatrix}
0&0\\
0&\eta_\phi(C)
\end{pmatrix}.
\]
Taking
\[
S=S(0,0,C),
\qquad
T=S(0,0,C')
\]
in \eqref{eq:psi-master} shows that $\eta_\phi$ satisfies
\eqref{eq:T-ab-equation}, hence $\eta_\phi\in\mathcal T_{a,b}$.

Uniqueness follows from
\[
L_\phi=\phi_{21}(S(I_{\mathcal E},0,0)),
\qquad
\eta_\phi(C)
=
\bigl((\phi-\delta_{L_\phi})(S(0,0,C))\bigr)_{22}.
\]
\end{proof}

\begin{proposition}\label{prop:quotient-tangent-inclusion}
Assume $2\leq k\leq n-2$ and let $(a,b)$ be one of the
choices in \eqref{eq:ab-choices}. Under
\eqref{eq:quotient-Grassmann-tangent},
\[
T_{\mathcal C_{a,b}}\mathfrak C_m
\subseteq
\mathcal T_{a,b}.
\]
\end{proposition}

\begin{proof}
In the graph chart determined by
$M_m(\mathbb C)=\mathcal C_{a,b}\oplus\mathcal Z_{a,b}$, a first-order
deformation represented by
$\eta\in\operatorname{Hom}(\mathcal C_{a,b},\mathcal Z_{a,b})$ has elements
$C+t\eta(C)$ modulo $t^2$. For $C,C'\in\mathcal C_{a,b}$,
\[
[C+t\eta(C),C'+t\eta(C')]
=
t\bigl([\eta(C),C']+[C,\eta(C')]\bigr)
\pmod{t^2},
\]
because $[C,C']=0$. The differentials of the defining commuting equations
therefore vanish on every tangent vector to $\mathfrak C_m$, which is exactly
the asserted inclusion.
\end{proof}

\begin{corollary}\label{cor:ambient-tangent-injection}
Assume $2\leq k\leq n-2$ and let $(a,b)$ be one of the
choices in \eqref{eq:ab-choices}. The assignment in
Proposition~\ref{prop:tangent-splitting} is a linear
injection
\[
T_{\mathcal A_{k,a,b}}\mathfrak X_{n,k}
\hookrightarrow
\operatorname{Hom}(\mathcal E,\mathcal Q)\oplus\mathcal T_{a,b},
\qquad
\phi\longmapsto(L_\phi,\eta_\phi).
\]
Consequently,
\[
\dim T_{\mathcal A_{k,a,b}}\mathfrak X_{n,k}
\leq
k(n-k)+\dim\mathcal T_{a,b}.
\]
\end{corollary}

\begin{proof}
Linearity follows from the construction of $L_\phi$ and $\eta_\phi$, and
injectivity is the uniqueness assertion in
Proposition~\ref{prop:tangent-splitting}. The dimension estimate follows from
$\dim\operatorname{Hom}(\mathcal E,\mathcal Q)=k(n-k)$.
\end{proof}

\begin{remark}
It remains to determine $\mathcal T_{a,b}$. When $m\geq4$,
Section~\ref{sec:m-ge-4} proves that every element of
$\mathcal T_{a,b}$ is induced by conjugation on the quotient. When
$m\in\{2,3\}$, Section~\ref{sec:m-2-3} shows that
$\mathcal T_{a,b}$ is the full tangent space to the smooth commuting locus
$\mathfrak C_m$ at $\mathcal C_{a,b}$.
\end{remark}

\section[The case n-k at least 4]{\texorpdfstring{The case $n-k\geq4$}{The case n-k at least 4}}\label{sec:m-ge-4}

In this section, assume $2\leq k\leq n-2$ and $m\geq4$, and fix
one of the choices $(a,b)$ in \eqref{eq:ab-choices}. By
Corollary~\ref{cor:ambient-tangent-injection}, it remains to determine the
space $\mathcal T_{a,b}$.

\subsection[Solving the equations for R, U, and V]{\texorpdfstring{Solving the equations for $R$, $U$, and $V$}{Solving the equations for R, U, and V}}

\begin{lemma}\label{lem:RUV-equations}
Assume $2\leq k\leq n-2$ and $m\geq4$, and let $(a,b)$ be one of
the choices in \eqref{eq:ab-choices}. Suppose that
\[
R\colon M_{a\times b}(\mathbb C)\to M_{b\times a}(\mathbb C)\quad 
U\colon M_{a\times b}(\mathbb C)\to M_a(\mathbb C),
\quad
V\colon M_{a\times b}(\mathbb C)\to M_b(\mathbb C)
\]
are linear maps satisfying \eqref{eq:R-left} and \eqref{eq:UV-identity} for all
$Y,Z\in M_{a\times b}(\mathbb C)$. Then $R=0$, and there exist
$N\in M_{b\times a}(\mathbb C)$ and a linear functional
$\tau\colon M_{a\times b}(\mathbb C)\to\mathbb C$ such that
\begin{equation}\label{eq:UV-solution}
U(Y)=-YN+\tau(Y)I_a,
\qquad
V(Y)=NY+\tau(Y)I_b.
\end{equation}
If, in addition,
\begin{equation}\label{eq:UV-trace}
\operatorname{tr}U(Y)+\operatorname{tr}V(Y)=0
\qquad(Y\in M_{a\times b}(\mathbb C)),
\end{equation}
then $\tau=0$.
\end{lemma}

\begin{proof}
Write $E_{i\alpha}$ for the standard matrix units of
$M_{a\times b}(\mathbb C)$. Set
$Y=E_{i\alpha}$ and $Z=E_{j\beta}$ with $j\neq i$. In
\eqref{eq:R-left}, the left-hand side is supported only in row $j$, while the
right-hand side is supported only in row $i$. Hence both sides vanish. Since
$\beta$ is arbitrary, every row of $R(E_{i\alpha})$ vanishes. Thus
$R(E_{i\alpha})=0$ for all $i,\alpha$, and therefore $R=0$. Write
\[
U(E_{i\alpha})=(u^{i\alpha}_{pq}),
\qquad
V(E_{i\alpha})=(v^{i\alpha}_{\beta\gamma}).
\]
Substituting $Y=E_{i\alpha}$ and $Z=E_{j\beta}$ into
\eqref{eq:UV-identity} and taking the $(p,\gamma)$-entry gives
\begin{equation}\label{eq:coordinate-identity}
u^{i\alpha}_{pj}\delta_{\gamma\beta}
-\delta_{pj}v^{i\alpha}_{\beta\gamma}
+\delta_{pi}v^{j\beta}_{\alpha\gamma}
-u^{j\beta}_{pi}\delta_{\gamma\alpha}=0.
\end{equation}
Fix $j\neq i$ and put $p=j$. Then
\begin{equation}\label{eq:v-first}
v^{i\alpha}_{\beta\gamma}
=
u^{i\alpha}_{jj}\delta_{\gamma\beta}
-u^{j\beta}_{ji}\delta_{\gamma\alpha}.
\end{equation}
Choose $\beta_0\neq\alpha$ and put
$\beta=\gamma=\beta_0$ in \eqref{eq:v-first}. It follows that
\[
u^{i\alpha}_{jj}=v^{i\alpha}_{\beta_0\beta_0},
\]
so this scalar is independent of $j\neq i$. Define
\[
\tau_{i\alpha}=u^{i\alpha}_{jj}
\qquad(j\neq i).
\]
For fixed $\beta$ and $i$, choose $\alpha_0\neq\beta$, use
$Y=E_{i\alpha_0}$ in \eqref{eq:v-first}, and put
$\gamma=\alpha_0$. Then
\[
-u^{j\beta}_{ji}=v^{i\alpha_0}_{\beta\alpha_0},
\]
which is independent of $j\neq i$. Define
\[
n_{\beta i}=-u^{j\beta}_{ji}
\qquad(j\neq i).
\]
Equation \eqref{eq:v-first} becomes
\begin{equation}\label{eq:v-solution-units}
v^{i\alpha}_{\beta\gamma}
=
\tau_{i\alpha}\delta_{\beta\gamma}
+n_{\beta i}\delta_{\alpha\gamma}.
\end{equation}
Choose $\alpha\neq\beta$ and put $\gamma=\alpha$ in
\eqref{eq:coordinate-identity}. Using
\eqref{eq:v-solution-units} gives
\[
u^{j\beta}_{pi}
=-\delta_{pj}n_{\beta i}+\delta_{pi}\tau_{j\beta}
\qquad(j\neq i).
\]
After relabelling,
\begin{equation}\label{eq:u-off-column}
u^{i\alpha}_{pq}
=-\delta_{pi}n_{\alpha q}+\delta_{pq}\tau_{i\alpha}
\qquad(q\neq i).
\end{equation}
It remains to determine the column $q=i$. Put $j=i$ in
\eqref{eq:coordinate-identity}. If $p\neq i$, choose
$\gamma=\beta\neq\alpha$; then $u^{i\alpha}_{pi}=0$. If $p=i$, choose again
$\beta\neq\alpha$ and substitute \eqref{eq:v-solution-units}. The
coefficients of $\delta_{\gamma\beta}$ and
$\delta_{\gamma\alpha}$ give
\[
u^{i\alpha}_{ii}=\tau_{i\alpha}-n_{\alpha i}.
\]
Together with \eqref{eq:u-off-column}, this yields
\begin{equation}\label{eq:u-solution-units}
u^{i\alpha}_{pq}
=-\delta_{pi}n_{\alpha q}+\delta_{pq}\tau_{i\alpha}
\end{equation}
for all $p,q$. Let
\[
N=(n_{\alpha q})\in M_{b\times a}(\mathbb C),
\qquad
\tau(E_{i\alpha})=\tau_{i\alpha},
\]
and extend $\tau$ linearly. Equations
\eqref{eq:v-solution-units} and \eqref{eq:u-solution-units} give
\eqref{eq:UV-solution}. Finally, for every matrix unit,
\[
\operatorname{tr}U(E_{i\alpha})
+
\operatorname{tr}V(E_{i\alpha})
=
\bigl(-n_{\alpha i}+a\tau_{i\alpha}\bigr)
+
\bigl(n_{\alpha i}+b\tau_{i\alpha}\bigr)=(a+b)\tau_{i\alpha}.
\]
Thus \eqref{eq:UV-trace} forces every $\tau_{i\alpha}$ to vanish.
\end{proof}

\subsection[The space T(a,b)]{\texorpdfstring{The space $\mathcal T_{a,b}$}{The space T(a,b)}}

For $N\in M_{b\times a}(\mathbb C)$, define
\[
\eta_N
\left(
\begin{pmatrix}
\lambda I_a&D\\
0&\lambda I_b
\end{pmatrix}
\right)
=
\begin{pmatrix}
-DN&0\\
0&ND
\end{pmatrix}.
\]
The identity $\operatorname{tr}(DN)=\operatorname{tr}(ND)$ shows that
$\eta_N$ takes values in $\mathcal Z_{a,b}$. It is the infinitesimal
conjugation direction induced by
$\left(\begin{smallmatrix}0&0\\N&0\end{smallmatrix}\right)$ on
$\mathbb C^a\oplus\mathbb C^b$.

\begin{proposition}\label{prop:T-ab-rigid}
Assume $2\leq k\leq n-2$ and $m\geq4$, and let $(a,b)$ be one of
the choices in \eqref{eq:ab-choices}. Then
\[
\mathcal T_{a,b}
=
\{\eta_N:N\in M_{b\times a}(\mathbb C)\}.
\]
In particular, $\dim\mathcal T_{a,b}=ab$.
\end{proposition}

\begin{proof}
Let $\eta\in\mathcal T_{a,b}$. Proposition~\ref{prop:quotient-identities}
produces maps $R,U,V$ satisfying the hypotheses of
Lemma~\ref{lem:RUV-equations}, including the trace condition \eqref{eq:UV-trace}. The lemma
gives a unique $N\in M_{b\times a}(\mathbb C)$ such that
\[
R=0,
\qquad
U(D)=-DN,
\qquad
V(D)=ND.
\]
Thus $\eta=\eta_N$. Conversely, each $\eta_N$ is obtained by differentiating
conjugation and therefore satisfies \eqref{eq:T-ab-equation}. The assignment
$N\mapsto\eta_N$ is injective, so $\dim\mathcal T_{a,b}=ab$.
\end{proof}

\subsection[The tangent space at A(k,a,b)]{\texorpdfstring{The tangent space at $\mathcal A_{k,a,b}$}{The tangent space at A(k,a,b)}}

\begin{theorem}\label{thm:tangent-equality-mge4}
Assume $2\leq k\leq n-2$ and $m\geq4$, and let $(a,b)$ be one of
the choices in \eqref{eq:ab-choices}. Then
\[
T_{\mathcal A_{k,a,b}}\mathfrak X_{n,k}
=
T_{\mathcal A_{k,a,b}}\mathcal O_{k,a,b}.
\]
\end{theorem}

\begin{proof}
Let $\phi\in T_{\mathcal A_{k,a,b}}\mathfrak X_{n,k}$. By
Proposition~\ref{prop:tangent-splitting}, write
\[
L_\phi=
\begin{pmatrix}
L\\
M
\end{pmatrix}
\colon\mathcal E\longrightarrow\mathcal F\oplus\mathcal G
\]
and let $\eta_\phi\in\mathcal T_{a,b}$ be the residual quotient
deformation. Proposition~\ref{prop:T-ab-rigid} gives a unique
$N\in M_{b\times a}(\mathbb C)$ such that $\eta_\phi=\eta_N$. Hence
$\phi=\delta_{L,M,N}$. Conversely, conjugation preserves
$\mathfrak X_{n,k}$, so every $\delta_{L,M,N}$ is tangent. The equality with
the orbit tangent space follows from
Proposition~\ref{prop:orbit-Akab-tangent}.
\end{proof}

\section[The cases n-k=2 and n-k=3]{\texorpdfstring{The cases $n-k=2$ and $n-k=3$}{The cases n-k=2 and n-k=3}}\label{sec:m-2-3}

In this section, assume $2\leq k\leq n-2$ and
$m\in\{2,3\}$, and fix one of the choices $(a,b)$ in
\eqref{eq:ab-choices}.
In these dimensions, not every tangent vector comes from conjugation. We
first give an example, then compute the tangent space of $\mathfrak C_m$, and
finally define the families that contain these additional tangent directions.

\subsection{A tangent vector not induced by conjugation}

\begin{remark}[A tangent vector not induced by conjugation]
Already for $m=2$, varying the commuting subspace on the quotient produces tangent directions
that are not induced by conjugation. Let
\[
\mathcal C_0=\operatorname{span}(I_2,E_{12})\subseteq M_2(\mathbb C)
\]
and
\[
\mathcal C_t
=
\operatorname{span}(I_2,E_{12}+tE_{21}).
\]
Each $\mathcal C_t$ is commuting, so the derivative at $t=0$ is a tangent
direction in $T_{\mathcal C_0}\mathfrak C_2$ represented by
$E_{21}$ modulo $\mathcal C_0$. On the other hand, for
$X=\left(\begin{smallmatrix}\alpha&\beta\\\gamma&\delta\end{smallmatrix}\right)$,
\[
[X,E_{12}]
=
\begin{pmatrix}
-\gamma&\alpha-\delta\\
0&\gamma
\end{pmatrix},
\]
which has zero $(2,1)$-entry even modulo $\mathcal C_0$. Thus this tangent
direction is not induced by conjugation. This already shows, for $m=2$, that
a single conjugacy orbit cannot account for all tangent directions. In the
remainder of this section we therefore allow the commuting subspace induced on
the quotient to vary.
\end{remark}

\subsection[Commuting subspaces of M2 and M3]{\texorpdfstring{Commuting subspaces of $M_2(\mathbb C)$ and $M_3(\mathbb C)$}{Commuting subspaces of M2 and M3}}

\begin{proposition}\label{prop:Cm-geometry}
Assume $2\leq k\leq n-2$ and $m\in\{2,3\}$. If $m=2$, then
\[
\mathfrak C_m
\cong
\operatorname{Gr}(1,\mathfrak{sl}_2)
=
\mathbb P(\mathfrak{sl}_2)
\cong
\mathbb P^2.
\]
If $m=3$, then
\[
\mathfrak C_m
\cong
\left\{
\mathfrak a\in\operatorname{Gr}(2,\mathfrak{sl}_3):
[\mathfrak a,\mathfrak a]=0
\right\}.
\]
Consequently, $\mathfrak C_m$ is smooth and irreducible, of dimension $2$
when $m=2$ and of dimension $6$ when $m=3$.
\end{proposition}
\begin{proof}
By Theorem~\ref{thm:commuting-classification}, every
commuting subspace of $M_m(\mathbb C)$ has dimension at most $s_m$. If
$I_m\notin\mathcal C\in\mathfrak C_m$, then
$\mathcal C\cap\mathbb C I_m=0$, and hence
\[
\dim(\mathcal C+\mathbb C I_m)
=
\dim\mathcal C+1
=
s_m+1.
\]
Since $I_m$ commutes with every matrix,
$\mathcal C+\mathbb C I_m$ would still be a commuting subspace, contradicting
the maximal dimension bound. Thus every point of $\mathfrak C_m$ contains the
fixed line $\mathbb C I_m$. Use the fixed splitting
\[
M_m(\mathbb C)=\mathbb C I_m\oplus\mathfrak{sl}_m.
\]
If $\mathcal C\in\mathfrak C_m$, then $\mathbb C I_m\subseteq\mathcal C$,
and the splitting gives
\[
\mathcal C
=
\mathbb C I_m\oplus
(\mathcal C\cap\mathfrak{sl}_m).
\]
In particular,
\[
\dim(\mathcal C\cap\mathfrak{sl}_m)=s_m-1.
\]
The Schubert subvariety of $s_m$-planes containing $\mathbb C I_m$ is
canonically isomorphic to
$\operatorname{Gr}(s_m-1,M_m/\mathbb C I_m)$, and the displayed splitting
identifies the quotient with $\mathfrak{sl}_m$. Under this identification, the
mutually inverse maps are
\[
\mathcal C\longmapsto\mathcal C\cap\mathfrak{sl}_m,
\qquad
\mathfrak a\longmapsto\mathbb C I_m\oplus\mathfrak a.
\]
It remains to translate the commuting condition under these maps. If
$A,B\in\mathfrak a$ and $\lambda,\mu\in\mathbb C$, then
\[
[\lambda I_m+A,\mu I_m+B]=[A,B],
\]
since scalar matrices commute with everything. Hence
$\mathbb C I_m\oplus\mathfrak a$ is commuting if and only if
$\mathfrak a$ is abelian. Therefore the preceding maps give an algebraic
isomorphism
\[
\mathfrak C_m
\cong
\left\{
\mathfrak a\in\operatorname{Gr}(s_m-1,\mathfrak{sl}_m):
[\mathfrak a,\mathfrak a]=0
\right\}.
\]

For $m=2$, one has $s_2-1=1$, and every line in
$\mathfrak{sl}_2$ is abelian. Thus
\[
\mathfrak C_2
\cong
\operatorname{Gr}(1,\mathfrak{sl}_2)
=
\mathbb P(\mathfrak{sl}_2)
\cong
\mathbb P^2,
\]
since $\dim\mathfrak{sl}_2=3$. In particular, $\mathfrak C_2$ is smooth and
irreducible of dimension $2$.

For $m=3$, one has $s_3-1=2$, so the preceding identification gives
\[
\mathfrak C_3
\cong
\left\{
\mathfrak a\in\operatorname{Gr}(2,\mathfrak{sl}_3):
[\mathfrak a,\mathfrak a]=0
\right\}.
\]
Iliev and Manivel prove that this locus is smooth, irreducible, and
six-dimensional
\cite[Section~1.3, Theorem~3.11, and Proposition~4.1]{IlievManivel}.
\end{proof}

\subsection[Tangent spaces at C(a,b)]{\texorpdfstring{Tangent spaces at $\mathcal C_{a,b}$}{Tangent spaces at C(a,b)}}

\begin{proposition}\label{prop:T-ab-dimension}
Assume $2\leq k\leq n-2$ and $m\in\{2,3\}$, and let $(a,b)$ be
one of the choices in \eqref{eq:ab-choices}. Then
\[
\dim\mathcal T_{a,b}
=
\begin{cases}
2,&m=2,\\
6,&m=3.
\end{cases}
\]
\end{proposition}

\begin{proof}
By the observation following \eqref{eq:T-ab-equation}, every
$\eta\in\mathcal T_{a,b}$ satisfies $\eta(I_m)=0$.

If $m=2$, then \eqref{eq:ab-choices} gives $(a,b)=(1,1)$. The space
$M_{1\times1}(\mathbb C)$ has dimension one, so the alternating bilinear equation
\eqref{eq:T-ab-equation} imposes no further condition. Since
\[
\dim\mathcal Z_{1,1}=4-2=2,
\]
one gets $\dim\mathcal T_{a,b}=2$.

Now suppose $m=3$. It suffices first to treat the choice $(a,b)=(1,2)$.
Identify $M_{1\times2}(\mathbb C)$ with the row space
$W=(\mathbb C^2)^*$. Write
\[
\eta(y)
=
\begin{pmatrix}
\nu(y)&0\\
R(y)&V(y)
\end{pmatrix},
\qquad
\nu(y)+\operatorname{tr}V(y)=0.
\]
By \eqref{eq:R-right}, for a basis $e_1,e_2$ of $W$,
\[
R(e_1)e_2=R(e_2)e_1.
\]
The two sides have disjoint column supports and
therefore force $R(e_1)=R(e_2)=0$. Hence $R=0$.

Set
\[
\Omega(y)=V(y)-\nu(y)I_2.
\]
Equation~\eqref{eq:UV-identity} becomes
\[
y\Omega(z)=z\Omega(y).
\]
Thus the bilinear map $B(y,z)=y\Omega(z)$ is symmetric. Conversely, every
symmetric bilinear map $B\colon W\times W\to W$ determines a unique linear
map $\Omega\colon W\to M_2(\mathbb C)$ by this formula. Therefore the
solution space has dimension
\[
\dim\operatorname{Hom}(\operatorname{Sym}^2W,W)=3\cdot2=6.
\]
The trace condition uniquely recovers
\[
\nu(y)=-\frac13\operatorname{tr}\Omega(y),
\qquad
V(y)=\Omega(y)+\nu(y)I_2.
\]
Hence $\dim\mathcal T_{1,2}=6$. Transposition followed by the block swap
identifies the $(2,1)$ and $(1,2)$ systems, so
$\dim\mathcal T_{a,b}=6$ for either choice in \eqref{eq:ab-choices} when
$m=3$.
\end{proof}

\begin{corollary}
Assume $2\leq k\leq n-2$ and $m\in\{2,3\}$, and let $(a,b)$ be
one of the choices in \eqref{eq:ab-choices}. Then
\[
T_{\mathcal C_{a,b}}\mathfrak C_m
=
\mathcal T_{a,b}
\]
under the graph-chart identification
\eqref{eq:quotient-Grassmann-tangent}.
\end{corollary}

\begin{proof}
Proposition~\ref{prop:quotient-tangent-inclusion} gives the inclusion from
left to right. Proposition~\ref{prop:Cm-geometry} shows that
$\mathfrak C_m$ is smooth at $\mathcal C_{a,b}$ and has dimension $2$ for
$m=2$ and $6$ for $m=3$. Proposition~\ref{prop:T-ab-dimension} gives the
same dimensions for $\mathcal T_{a,b}$. Hence the inclusion is an equality.
\end{proof}

\subsection[The families Y+ and Y-]{\texorpdfstring{The families $\mathfrak Y^+_{n,k}$ and $\mathfrak Y^-_{n,k}$}{The families Y+ and Y-}}

Recall the spaces $\mathcal V_{\mathcal E}(\mathcal C)$ and the families
$\mathfrak Y^\pm_{n,k}$ defined in \eqref{eq:V-EC} and
\eqref{eq:Y-minus}.

\begin{proposition}\label{prop:Y-geometry}
Assume $2\leq k\leq n-2$ and $m\in\{2,3\}$. Then the sets
$\mathfrak Y^+_{n,k}$ and $\mathfrak Y^-_{n,k}$ are closed irreducible
subvarieties of $\mathfrak X_{n,k}$, and
\[
\dim\mathfrak Y^+_{n,k}
=
\dim\mathfrak Y^-_{n,k}
=
k(n-k)+\dim\mathfrak C_m.
\]
\end{proposition}
\begin{proof}
Put \(B=\operatorname{Gr}(k,n)\), and let
\(\mathcal Q_{\mathrm{taut}}\) be the tautological quotient bundle on \(B\).
Consider
\[
\mathcal P^+_{n,k}
=
\left\{
(\mathcal E,\mathcal C)\in
\operatorname{Gr}_{B}
\bigl(s_m,\operatorname{End}(\mathcal Q_{\mathrm{taut}})\bigr):
[C,D]=0
\text{ for all }C,D\in\mathcal C
\right\}.
\]
On every open set \(U\subseteq B\) over which
\(\mathcal Q_{\mathrm{taut}}\) is trivial, one has
\[
\mathcal P^+_{n,k}|_U
\cong
U\times\mathfrak C_m.
\]
Thus \(\mathcal P^+_{n,k}\) is closed in $\operatorname{Gr}_{B}
\bigl(s_m,\operatorname{End}(\mathcal Q_{\mathrm{taut}})\bigr)$ and is
therefore projective. Since \(B\) and \(\mathfrak C_m\) are irreducible,
the local product description also gives
\[
\mathcal P^+_{n,k}\ \text{irreducible},
\qquad
\dim\mathcal P^+_{n,k}
=
k(n-k)+\dim\mathfrak C_m.
\]

Define
\[
\iota:\mathcal P^+_{n,k}
\longrightarrow
\operatorname{Gr}(d_{n,k},M_n(\mathbb C)),
\qquad
(\mathcal E,\mathcal C)
\longmapsto
\mathcal V_{\mathcal E}(\mathcal C).
\]
We verify directly that this assignment is algebraic. Fix a decomposition
\[
\mathbb C^n=\mathcal E_0\oplus\mathcal F_0,
\qquad
\dim\mathcal E_0=k,
\quad
\dim\mathcal F_0=m,
\]
and consider the standard affine chart consisting of the graphs
\[
\mathcal E_Z
=
\{x+Zx:x\in\mathcal E_0\},
\qquad
Z\in\operatorname{Hom}(\mathcal E_0,\mathcal F_0).
\]
Set
\[
g_Z=
\begin{pmatrix}
I_k&0\\
Z&I_m
\end{pmatrix},
\qquad
g_Z^{-1}
=
\begin{pmatrix}
I_k&0\\
-Z&I_m
\end{pmatrix}.
\]
Then \(g_Z\mathcal E_0=\mathcal E_Z\), and \(g_Z\) gives the corresponding
identification
\(\mathbb C^n/\mathcal E_0\cong\mathbb C^n/\mathcal E_Z\).

Let
\[
\mathcal K
=
\left\{
\begin{pmatrix}
A&X\\
0&0
\end{pmatrix}:
A\in M_k(\mathbb C),\
X\in M_{k\times m}(\mathbb C)
\right\}.
\]
Under the local identification
\(\mathcal P^+_{n,k}|_U\cong U\times\mathfrak C_m\), one has
\[
\iota(\mathcal E_Z,\mathcal C)
=
g_Z
\left(
\mathcal K
\oplus
\left\{
\begin{pmatrix}
0&0\\
0&C
\end{pmatrix}:
C\in\mathcal C
\right\}
\right)
g_Z^{-1}.
\]
The map
\[
\mathcal C
\longmapsto
\mathcal K
\oplus
\left\{
\begin{pmatrix}
0&0\\
0&C
\end{pmatrix}:
C\in\mathcal C
\right\}
\]
is a morphism between Grassmannians, and both \(g_Z\) and \(g_Z^{-1}\)
depend polynomially on \(Z\). Hence the displayed formula is regular on
the chart. These local formulas agree on overlaps because they all describe
the intrinsic space \(\mathcal V_{\mathcal E}(\mathcal C)\). Thus
\(\iota\) is a morphism.

Its image is \(\mathfrak Y^+_{n,k}\), which is closed because
\(\mathcal P^+_{n,k}\) is projective. Moreover,
\(\mathfrak Y^+_{n,k}\subseteq\mathfrak X_{n,k}\): if
\(S,T\in\mathcal V_{\mathcal E}(\mathcal C)\), then
\[
[\overline S,\overline T]=0
\]
on \(\mathbb C^n/\mathcal E\), and therefore
\[
\operatorname{im}[S,T]\subseteq\mathcal E,
\qquad
\operatorname{rank}[S,T]\leq k.
\]

It remains to determine the fibers of \(\iota\). For a subspace
\(\mathcal V=\mathcal V_{\mathcal E}(\mathcal C)\), put
\[
\mathcal E(\mathcal V)
=
\sum_{S,T\in\mathcal V}\operatorname{im}[S,T].
\]
Every commutator image is contained in \(\mathcal E\), so
\[
\mathcal E(\mathcal V)\subseteq\mathcal E.
\]
Conversely, since \(k\geq2\), choose a basis of \(\mathcal E\), and let
\(A_\ast,B_\ast\in\operatorname{End}(\mathcal E)\) be represented in this
basis by the matrices in \eqref{eq:AB-star}. Then
\([A_\ast,B_\ast]\) is invertible. Relative to any complement of
\(\mathcal E\), the operators
\[
S=
\begin{pmatrix}
A_\ast&0\\
0&0
\end{pmatrix},
\qquad
T=
\begin{pmatrix}
B_\ast&0\\
0&0
\end{pmatrix}
\]
belong to \(\mathcal V_{\mathcal E}(\mathcal C)\), and
\(\operatorname{im}[S,T]=\mathcal E\). Hence
\[
\mathcal E(\mathcal V)=\mathcal E.
\]
Thus \(\mathcal V\) determines \(\mathcal E\), and then it determines
\(\mathcal C\) by
\[
\mathcal C
=
\{\overline T:T\in\mathcal V\}
\subseteq
\operatorname{End}(\mathbb C^n/\mathcal E).
\]
Therefore \(\iota\) is injective. Its fibers are zero-dimensional, so the
fiber-dimension theorem gives
\[
\dim\mathfrak Y^+_{n,k}
=
\dim\mathcal P^+_{n,k}
=
k(n-k)+\dim\mathfrak C_m.
\]
The image is irreducible because \(\mathcal P^+_{n,k}\) is irreducible.

Finally, transposition is an algebraic automorphism of the ambient
Grassmannian, so the same conclusions hold for
\(\mathfrak Y^-_{n,k}\).
\end{proof}

\section{Irreducible components}\label{sec:components}

\subsection{A dimension lemma}

\begin{lemma}\label{lem:component-dimension-criterion}
Let $X$ be an algebraic set over $\mathbb C$, let $p\in X$, and let
$Y\subseteq X$ be a closed irreducible subvariety containing $p$. If
\[
\dim_{\mathbb C}T_pX\leq\dim Y,
\]
then $p$ is a nonsingular point of $X$, lies on a unique irreducible
component of $X$, and that component is $Y$.
\end{lemma}

\begin{proof}
The standard local-dimension inequalities give
\[
\dim Y
\leq
\dim_pX
\leq
\dim_{\mathbb C}T_pX
\leq
\dim Y.
\]
Hence equality holds throughout. Thus the local ring
$\mathcal O_{X,p}$ has Krull dimension equal to its embedding dimension,
so it is regular. Therefore $p$ is nonsingular. Since a regular local ring
is a domain, $p$ lies on a unique irreducible component $Z$ of $X$.

Since $Y$ is irreducible and contains $p$, it is contained in $Z$. Moreover,
\[
\dim Z=\dim_pX=\dim Y.
\]
A proper closed irreducible subvariety of $Z$ has strictly smaller dimension,
so $Y=Z$.
\end{proof}

\subsection[Components containing A(k,a,b)]{\texorpdfstring{Components containing $\mathcal A_{k,a,b}$}{Components containing A(k,a,b)}}

\begin{proposition}\label{prop:model-component-mge4}
Assume $2\leq k\leq n-2$ and $m\geq4$, and let $(a,b)$ be one of
the choices in \eqref{eq:ab-choices}. Then the unique irreducible component of
$\mathfrak X_{n,k}$ through $\mathcal A_{k,a,b}$ is
$\mathcal O_{k,a,b}$.
\end{proposition}

\begin{proof}
By Proposition~\ref{prop:orbit-Akab}, the orbit $\mathcal O_{k,a,b}$ is a
closed irreducible subvariety of $\mathfrak X_{n,k}$ through
$\mathcal A_{k,a,b}$ and has dimension $k(n-k)+ab$. Theorem~\ref{thm:tangent-equality-mge4} and
Proposition~\ref{prop:orbit-Akab} give
\[
\dim T_{\mathcal A_{k,a,b}}\mathfrak X_{n,k}
=
\dim\mathcal O_{k,a,b}.
\]
Lemma~\ref{lem:component-dimension-criterion}, applied with
$Y=\mathcal O_{k,a,b}$, gives the result.
\end{proof}

\begin{proposition}\label{prop:model-component-m23}
Assume $2\leq k\leq n-2$ and $m\in\{2,3\}$, and let $(a,b)$ be
one of the choices in \eqref{eq:ab-choices}. Then the unique irreducible
component of $\mathfrak X_{n,k}$ through $\mathcal A_{k,a,b}$ is
$\mathfrak Y^+_{n,k}$.
\end{proposition}

\begin{proof}
Put
\[
D=k(n-k)+\dim\mathfrak C_m.
\]
By Proposition~\ref{prop:Y-geometry},
$\mathfrak Y^+_{n,k}$ is a closed irreducible subvariety of
$\mathfrak X_{n,k}$ through $\mathcal A_{k,a,b}$ and has dimension $D$.
Corollary~\ref{cor:ambient-tangent-injection} gives
\[
\dim T_{\mathcal A_{k,a,b}}\mathfrak X_{n,k}
\leq
k(n-k)+\dim\mathcal T_{a,b}.
\]
By Propositions~\ref{prop:Cm-geometry} and~\ref{prop:T-ab-dimension},
$\dim\mathcal T_{a,b}=\dim\mathfrak C_m$. Hence
\[
\dim T_{\mathcal A_{k,a,b}}\mathfrak X_{n,k}
\leq D
=
\dim\mathfrak Y^+_{n,k}.
\]
Lemma~\ref{lem:component-dimension-criterion}, applied with
$Y=\mathfrak Y^+_{n,k}$, gives the result.
\end{proof}

\subsection{Distinctness of the components}

\begin{proposition}\label{prop:components-distinct}
Assume $2\leq k\leq n-2$.
\begin{enumerate}
\item If $m\geq4$, then the orbits
$\mathcal O_{k,a,b}$ and $\mathcal O_{k,a,b}^{\mathsf T}$, as $(a,b)$
ranges over the choices in \eqref{eq:ab-choices}, are pairwise distinct.
\item If $m\in\{2,3\}$, then
$\mathfrak Y^+_{n,k}\cap\mathfrak Y^-_{n,k}=\varnothing$.
\end{enumerate}
\end{proposition}

\begin{proof}
We first distinguish the orbits $\mathcal O_{k,a,b}$ for different ordered
pairs $(a,b)$. Suppose that
\[
\mathcal O_{k,a,b}=\mathcal O_{k,a',b'}.
\]
Then $\mathcal A_{k,a,b}$ and $\mathcal A_{k,a',b'}$ are conjugate. Let
$J$ and $J'$ be their Jacobson radicals. Conjugation carries $J^2$ onto
$(J')^2$, so it preserves the dimension of the common kernel of the matrices
in the radical square. By \eqref{eq:intrinsic-flag-general},
\[
\dim\bigcap_{R\in J^2}\ker R=k+a,
\qquad
\dim\bigcap_{R\in (J')^2}\ker R=k+a'.
\]
Hence $a=a'$, and then $b=b'$ because $a+b=a'+b'=m$. Thus the orbits
$\mathcal O_{k,a,b}$ belonging to the two choices in
\eqref{eq:ab-choices} are distinct whenever the choices are distinct. Their
transpose orbits are also distinct.

It remains to distinguish the spaces above from their transpose families.
For a linear subspace $\mathcal V\subseteq M_n(\mathbb C)$, the integer
\begin{equation}\label{eq:commutator-image-dimension}
\dim_{\mathbb C}
\left(
\sum_{S,T\in\mathcal V}\operatorname{im}[S,T]
\right)
\end{equation}
is unchanged by similarity. Let
$\mathcal V=\mathcal V_{\mathcal E}(\mathcal C)$ be a space in
$\mathfrak Y^+_{n,k}$. Every commutator in $\mathcal V$ induces zero on
$\mathbb C^n/\mathcal E$, so every commutator image is contained in
$\mathcal E$. On the other hand, the matrices acting as $A_\ast$ and $B_\ast$ on
$\mathcal E$ and as zero on a complement belong to $\mathcal V$, and their
commutator has image $\mathcal E$. Therefore
\eqref{eq:commutator-image-dimension} equals $k$ for every
$\mathcal V\in\mathfrak Y^+_{n,k}$.

Now consider $\mathcal V^{\mathsf T}$. Choose
$g\in\operatorname{GL}_n(\mathbb C)$ such that
$g^{-1}\mathcal E$ is the first $k$-dimensional coordinate subspace, and put
$\mathcal V_0=g^{-1}\mathcal Vg$. Then $\mathcal V_0$ has the corresponding
upper block form, and
\[
\mathcal V_0^{\mathsf T}
=
g^{\mathsf T}\mathcal V^{\mathsf T}g^{-\mathsf T}.
\]
Thus $\mathcal V^{\mathsf T}$ is similar to
$\mathcal V_0^{\mathsf T}$, and it is enough to compute in this block form.
The space $\mathcal V_0^{\mathsf T}$ contains
\[
S_0=
\begin{pmatrix}
I_k&0\\
0&0
\end{pmatrix},
\qquad
T_Y=
\begin{pmatrix}
0&0\\
Y&0
\end{pmatrix}
\quad
\bigl(Y\in M_{m\times k}(\mathbb C)\bigr).
\]
Since $[S_0,T_Y]=-T_Y$, the images of these commutators span the last
$m$ coordinate directions. The space $\mathcal V_0^{\mathsf T}$ also contains
\[
\begin{pmatrix}
A_\ast&0\\
0&0
\end{pmatrix},
\qquad
\begin{pmatrix}
B_\ast&0\\
0&0
\end{pmatrix},
\]
whose commutator has image equal to the first $k$ coordinate directions.
Thus \eqref{eq:commutator-image-dimension} equals $n$ for every space in
$\mathfrak Y^-_{n,k}$.

Since $k<n$, the sets $\mathfrak Y^+_{n,k}$ and
$\mathfrak Y^-_{n,k}$ are disjoint. Since
$\mathcal A_{k,a,b}=\mathcal V_{\mathcal E}(\mathcal C_{a,b})$ for the
standard $k$-plane $\mathcal E$, the same calculation shows that no
$\mathcal O_{k,a,b}$ equals any transpose orbit
$\mathcal O_{k,a',b'}^{\mathsf T}$. This completes both parts.
\end{proof}

\begin{proof}[Proof of Theorem~\ref{thm:components-intro}]
Let $Z$ be an irreducible component of $\mathfrak X_{n,k}$. By
Proposition~\ref{prop:component-Akab}, after possibly replacing $Z$ by its
transpose image, $Z$ contains a point in the conjugacy orbit of some
$\mathcal A_{k,a,b}$. Since $Z$ is stable under conjugation, we may assume
that $\mathcal A_{k,a,b}\in Z$.

If $m\geq4$, Proposition~\ref{prop:model-component-mge4} gives
$Z=\mathcal O_{k,a,b}$. Undoing the optional transposition shows that every
component is one of the orbits $\mathcal O_{k,a,b}$ or
$\mathcal O_{k,a,b}^{\mathsf T}$. Conversely,
Proposition~\ref{prop:model-component-mge4} and transposition show that every
such orbit is an irreducible component. Proposition~\ref{prop:components-distinct}
shows that they are pairwise distinct. Since distinct conjugacy orbits are
disjoint, these components are pairwise disjoint. When $m$ is even, the two choices in
\eqref{eq:ab-choices} coincide, so there are exactly two components. When
$m$ is odd, the two choices are distinct, so there are exactly four.
Proposition~\ref{prop:orbit-Akab} and transposition give the flag-variety
description and the dimension, while
Theorem~\ref{thm:tangent-equality-mge4} gives the stated tangent-space
equality.

If $m\in\{2,3\}$, Proposition~\ref{prop:model-component-m23} gives
$Z=\mathfrak Y^+_{n,k}$. Undoing the optional transposition shows that every
component is either $\mathfrak Y^+_{n,k}$ or $\mathfrak Y^-_{n,k}$.
Proposition~\ref{prop:model-component-m23} and transposition show that both
are irreducible components, and
Proposition~\ref{prop:components-distinct} shows that they are disjoint.
Thus there are exactly two components. Proposition~\ref{prop:Y-geometry}
gives their dimension as
\[
k(n-k)+\dim\mathfrak C_m,
\]
and Proposition~\ref{prop:Cm-geometry} gives
$\dim\mathfrak C_2=2$ and $\dim\mathfrak C_3=6$.
\end{proof}

\section{Proof of the main theorem}\label{sec:main-proof}

\begin{proof}[Proof of Theorem~\ref{thm:main}]
Let $n\geq1$, let $0\leq k<n$, and let
$\mathcal V\subseteq M_n(\mathbb C)$ satisfy the hypotheses of
Theorem~\ref{thm:main}.

If $k=0$, the assertion follows from
Theorem~\ref{thm:commuting-classification}. If $k=1$ or
$k=n-1$, the assertion follows from
Theorem~\ref{thm:ORS-results}.

It remains to consider $2\leq k\leq n-2$. Then
$\mathcal V\in\mathfrak X_{n,k}$. If $m\geq4$, the first part of
Theorem~\ref{thm:components-intro} shows that, after a similarity and possibly
transposition, $\mathcal V=\mathcal A_{k,a,b}$ for one of the choices in
\eqref{eq:ab-choices}. This is exactly \eqref{eq:main-block-space} with
$\mathcal C=\mathcal C_{a,b}$.

Now suppose $m\in\{2,3\}$. The second part of
Theorem~\ref{thm:components-intro} shows that, after possibly transposing,
\[
\mathcal V=\mathcal V_{\mathcal E}(\mathcal C)
\]
for a $k$-plane $\mathcal E$ and an $s_m$-dimensional commuting subspace
$\mathcal C\subseteq\operatorname{End}(\mathbb C^n/\mathcal E)$. By
Theorem~\ref{thm:commuting-classification},
$\mathcal C$ is, up to similarity, one of the possibilities in
\ref{item:C-m3} or \ref{item:C-m2}, according as $m=3$ or $m=2$.
A similarity of the quotient lifts to a similarity of
$\mathbb C^n$ preserving $\mathcal E$. After choosing a complement of
$\mathcal E$, the space $\mathcal V_{\mathcal E}(\mathcal C)$ has exactly
the form \eqref{eq:main-block-space}. The possibilities in
\ref{item:C-m3} and \ref{item:C-m2} are stable under transposition up to
similarity, and both orders of $(a,b)$ are allowed.
This proves the remaining cases.
\end{proof}

\section*{Declaration of AI use} OpenAI's ChatGPT was used during the development of this work for assistance with literature searches, exploration and verification of mathematical arguments, and manuscript preparation. All mathematical claims, proofs, citations, and the final text were independently checked and approved by the author, who takes full responsibility for the work.

\end{document}